\documentclass[12pt]{article}
\usepackage{graphicx} 
\usepackage{amsmath,amssymb,bbm, mathtools,enumerate}
\usepackage{amsthm}
\usepackage{fullpage}
\usepackage{xcolor}
\allowdisplaybreaks

\theoremstyle{plain}
\newtheorem{Theorem}{Theorem}[section]

\newtheorem{Lemma}[Theorem]{Lemma}

\newtheorem{Corollary}[Theorem]{Corollary}
\newtheorem{Remark}[Theorem]{Remark}

\def \N{\mathbb{N}}
\def \P{\mathbb{P}}
\def \R{\mathbb{R}}

\def \Nc{{\mathcal N}}
\def \Pc{{\mathcal P}}

\def \Wc{{\mathcal W}}

\title{Mean convergence rates for Gaussian-smoothed Wasserstein distances and classical Wasserstein distances}
\author{
Andrea COSSO\footnote{Universit\`a degli Studi di Milano; andrea.cosso@unimi.it}\quad\qquad Mattia MARTINI\footnote{Université C\^ote d'Azur, CNRS, Laboratoire J. A. Dieudonné; mattia.martini@univ-cotedazur.fr}\quad\qquad
Laura PERELLI\footnote{Universit\`a degli Studi di Milano; laura.perelli@unimi.it}}
\date{February 3, 2026}
\allowdisplaybreaks
\begin{document}
\maketitle

\begin{abstract}
\noindent We establish upper bounds for the expected $p$-th power of the Gaussian-smoothed $p$-Wasserstein distance between a probability measure $\mu$ and the corresponding empirical measure $\mu_N$, whenever $\mu$ has finite $q$-th moment for some $q>p$. This generalizes recent results that were valid only for $q>2p+2d$. We provide two distinct proofs of such a result. We also investigate the optimality of these bounds by establishing a lower bound of order $N^{-1/2-\varepsilon}$ for a probability measure possessing finite moments of all orders. Finally, we exploit a third upper bound for the Gaussian-smoothed $p$-Wasserstein distance to derive new convergence rates for the classical $p$-Wasserstein distance in the critical regime where $\mu$ has finite $p$-th moment but infinite moments of order $q > p$, covering for instance the case of Zygmund classes $L^p(\log L)^\alpha$.
\end{abstract}

\vspace{5mm}

distance.

\vspace{5mm}


\section{Introduction}

\emph{Background and motivation.} Let $d\geq1$ and let $\Pc(\R^d)$ denote the set of all probability measures on the Borel subsets of $\R^d$. For $p\geq1$, the Wasserstein distance $\Wc_p$ on $\Pc(\R^d)$ is defined as
\[
\Wc_p(\mu,\nu) \ = \ \inf\bigg\{\bigg(\int_{\R^d\times\R^d}|x-y|^p\,\pi(dx,dy)\bigg)^{1/p}\colon\pi\in\Pi(\mu,\nu)\bigg\}, \qquad \mu,\nu\in\Pc(\R^d),
\]
where $\Pi(\mu,\nu)$ is the family of all couplings of $\mu$ and $\nu$, that is the set of all probability measures on the Borel subsets of $\R^d\times\R^d$ with marginals $\mu$ and $\nu$. For further details on the Wasserstein distance, see for instance \cite{AGS08,CD18_I,Vi09}. Given $\mu\in\Pc(\R^d)$ and $N\geq1$, we consider the corresponding empirical measure
\[
\mu_N \ = \ \frac{1}{N}\sum_{k=1}^N \delta_{X_k},
\]
where $(X_k)_{k\geq1}$ is an i.i.d. sequence of $\mu$-distributed random variables. It is well-known (see \cite{Varad}) that Glivenko-Cantelli's theorem holds: $\mu_N$ weakly converges to $\mu$ almost surely as $N\rightarrow\infty$. Estimating the distance between $\mu_N$ and $\mu$ is a fundamental problem in probability, statistics, and machine learning, particularly when the Wasserstein distance is used. Significant contributions on this topic include \cite{Dudley,Ajtai,DobricYukich}, as well as more recent works \cite{BolleyGV,BoissardLeGouic,Dereich,FG15,BobkovLedoux,WeedBach,Lei,Lacker}, and book references such as \cite{CD18_I,ChewiWeedRigollet}. The Wasserstein distance between $\mu_N$ and $\mu$ is, however, often affected by the curse of dimensionality when $d>2p$ and the rate deteriorates to $N^{-p/d}$.

\vspace{2mm}

\noindent\emph{Gaussian smoothing.} To mitigate the curse of dimensionality, recent literature has focused on the \emph{Gaussian-smoothed Wasserstein distance} \cite{GoldfeldGreenewald,GGK20,GGNWP20,Goldfeld_et_al,Nietert_et_al,BlockJPR}.
Let $\mathcal{N}_\sigma$ denote the centered normal distribution with covariance $\sigma^2 I_d$. The density of $\Nc_\sigma$, denoted $\varphi^\sigma$, is given by
\begin{equation}\label{varphi}
\varphi^\sigma(x) \ = \ \frac{1}{(2\pi)^{d/2}\sigma^d}\,\text{e}^{-\frac{1}{2\sigma^2}|x|^2}, \qquad \forall\,x\in\R^d.
\end{equation}
We use the notation $\varphi$ instead of $\varphi^1$ when $\sigma=1$. The smoothed distance is defined as
\[
\mathcal{W}_p^{(\sigma)}(\mu,\nu) \ \coloneqq \ \mathcal{W}_p(\mu*\mathcal{N}_\sigma,\nu*\mathcal{N}_\sigma).
\]

\noindent\emph{Existing literature on rates for Gaussian-smoothed distance.} This distance was already used for deriving an upper bound for the classical 2-Wasserstein distance in \cite{HK94}, although such a distance was not formally defined at the time; on the other hand, smoothed empirical measures were first considered in \cite{Yukich} (here smoothing is not necessarily Gaussian), see also \cite[Chapter 2]{HighDimProbII}. Recently (see \cite[Theorem 2]{Nietert_et_al}) it was proved that
\begin{equation}
\mathbb E\big[\Wc_p^{(\sigma)}(\mu_N,\mu)\big] \ = \ O\bigg(\frac{1}{N^{1/(2p)}}\bigg)
\end{equation}
whenever
\[
\int_0^\infty r^{d+p-1}\sqrt{\P(|X|>r)}\,dr \ < \ \infty, \qquad\text{where } X\sim\mu,
\]
which holds when $\mu$ has finite $q$-th moment for $q>2p+2d$. In the case where $\mu$ is sub-Gaussian, i.e., there exists a constant $K>0$ such that
\[
\mathbb E\big[\text{e}^{\langle u,X-\mathbb E[X]\rangle}\big] \ \leq \ \text{e}^{\frac{1}{2}K^2|u|^2}, \quad \forall\,u\in\R^d,\qquad X\sim\mu,
\]
the rate of convergence can improve to $1/N$ (see \cite{BlockJPR} for a thorough analysis of this case).

\vspace{2mm}

\noindent\emph{Main contribution for the Gaussian-smoothed distance and methodology.} In this paper, we derive an upper bound for $\mathbb E[(\Wc_p^{(\sigma)}(\mu_N,\mu))^p]$ whenever $\mu$ has finite $q$-th moment, for any $q>p$, see Theorem \ref{T:RateSmooth}. Notice that, by Jensen's inequality, $\mathbb E[(\Wc_p^{(\sigma)}(\mu_N,\mu))]\leq\mathbb E[(\Wc_p^{(\sigma)}(\mu_N,\mu))^p]^{1/p}$, so that Theorem \ref{T:RateSmooth} gives a rate of convergence also for $\mathbb E[(\Wc_p^{(\sigma)}(\mu_N,\mu))]$. We provide two different proofs of Theorem \ref{T:RateSmooth}. First, we combine the original approach from \cite{HK94}, based on a Carlson-type inequality (that here we need to generalize, see Lemma \ref{L:Carlson}) with the Marcinkiewicz–Zygmund inequality (Lemma \ref{L:MZ_Ineq}). Second, we apply the dyadic partitioning argument (see for instance \cite{ChewiWeedRigollet}), which is crucial for obtaining sharp bounds in the classical Wasserstein distance case; in the present context, this argument simplifies significantly, and we combine it with the Marcinkiewicz–Zygmund inequality. Since $\Wc_p^{(\sigma)}(\mu,\nu)\leq\Wc_p(\mu,\nu)$, for every $\mu,\nu\in\mathcal P_p(\mathbb R^d)$, we can combine Theorem \ref{T:RateSmooth} with the well-known rates for the classical Wasserstein distance (see \cite[Theorem 1]{FG15}) to get a sharper result. In particular, Theorem \ref{T:RateSmooth} yields a better rate than \cite[Theorem 1]{FG15} whenever $p\leq d/2$ and $q>2dp/(d-p)$. Moreover, this rate achieves the dimension-free bound $N^{-1/2}$ when $q>2p+d$. See Remark \ref{R:FG} for a more detailed analysis of these rates.

\vspace{2mm}

\noindent\emph{Optimality of the convergence rates.} Complementing our analysis of upper bounds, we also investigate the sharpness of the derived rates for the Gaussian-smoothed Wasserstein distance. We establish a lower bound demonstrating that there exists a probability measure with moments of all orders such that, for any $\varepsilon > 0$, $\mathbb E[(\Wc_p^{(\sigma)}(\mu_N,\mu))^p]$ behaves at least as $N^{-1/2-\varepsilon}$.

\vspace{2mm}

\noindent\emph{Application to the Wasserstein distance.} It is easy to see (Lemma \ref{L:Bound}) that
\[
\Wc_p(\mu_N,\mu) \ \leq \ C_{p,d}\,\sigma+\Wc_p^{(\sigma)}(\mu_N,\mu),
\]
so, by taking $\sigma=\sigma_N\rightarrow0$, we can obtain an upper bound for $\Wc_p$ from the bound for $\Wc_p^{(\sigma)}$. We derive this result using another upper bound for $\mathbb E[(\Wc_p^{(\sigma)}(\mu_N,\mu))^p]$, established via a truncation argument. Such a methodology yields an upper bound in the critical case where $\mu$ possesses no finite moments of order $q>p$ (see Theorem \ref{T:RateG}), for which there are no results in the literature, except in the one-dimensional case, for which we refer to \cite{BGM,BobkovLedoux}. In particular, we specialize Theorem \ref{T:RateG} to the case where $\mu$ satisfies
\[
\int_{\mathbb R^d}|x|^p\big(\log(1+|x|)\big)^\alpha\,\mu(dx) \ < \ \infty,
\]
for some $p\in[1,\infty)$ and $\alpha>0$. Equivalently, if $X\sim\mu$ then $X$ belongs to the Zygmund space $L^p(\log L)^\alpha$. In such a context, we obtain the rate of convergence \eqref{UpperBoundWassZygmund}.

\vspace{2mm}

\noindent\emph{Organization.} Section 2 presents the main results for the Gaussian-smoothed distance, including the two proofs of the upper bound and the derivation of the lower bound. Section 3 presents the application to the classical Wasserstein distance for distributions with critical moment behavior.

\section{Gaussian-smoothed Wasserstein distance}

For every $p\geq1$ we denote by $\Pc_p(\R^d)$ the subset of $\Pc(\R^d)$ of all probability measures having finite $p$-th moment. For $\mu\in\Pc(\R^d)$ we denote
\[
M_p(\mu) \ = \ \bigg(\int_{\R^d}|x|^p\,\mu(dx)\bigg)^{1/p}.
\]
We recall from \cite[Proposition 1]{Nietert_et_al} and \cite[Lemma 4.2]{MKV_Uniq} that $(\Pc_p(\R^d),\Wc_p^{(\sigma)})$ is a complete, separable, metric space with the same topology as $(\Pc_p(\R^d),\Wc_p)$. We also have the following result.

\begin{Lemma}\label{L:Bound}
For any $p\geq1$, $\sigma>0$ and $\mu,\nu\in\Pc_p(\R^d)$, it holds that
\[
\Wc_p(\mu,\nu) \ \leq \ C_{p,d}\,\sigma+\Wc_p^{(\sigma)}(\mu,\nu),
\]
with $C_{p,d}=2^{3/2}\left(\frac{\Gamma(\frac{p+d}{2})}{\Gamma(\frac{d}{2})}\right)^{1/p}$, where $\Gamma(\cdot)$ denotes the usual gamma function.
\end{Lemma}
\begin{proof}
Consider random variables $X\sim\mu$, $Y\sim\nu$, $Z\sim\Nc_\sigma$, with $Z$ independent of $(X,Y)$. Notice that $(X,X+Z)$ is a coupling of $\mu$ and $\mu*\Nc_\sigma$. Similarly,  $(Y,Y+Z)$ is a coupling of $\nu$ and $\nu*\Nc_\sigma$. The triangle inequality, then, gives
\begin{align*}
\Wc_p(\mu,\nu) \ &\leq \ \Wc_p(\mu,\mu*\Nc_\sigma) + \Wc_p(\nu,\nu*\Nc_\sigma) + \Wc_p(\mu*\Nc_\sigma,\nu*\Nc_\sigma) \\
&\leq \ \|X-(X+Z)\|_{L^p} + \|Y-(Y+Z)\|_{L^p} + \Wc_p(\mu*\Nc_\sigma,\nu*\Nc_\sigma) \\
&= \ 2\|Z\|_{L^p} + \Wc_p(\mu*\Nc_\sigma,\nu*\Nc_\sigma) \ = \ 2\sigma M_p(\Nc_1) + \Wc_p(\mu*\Nc_\sigma,\nu*\Nc_\sigma).
\end{align*}
Thus, the claim follows from the equalities (here we use \cite[formula 4.642]{GradshteinRyzhik})
\begin{equation}\label{Moments_N_1}
M_p^p(\Nc_1) \ = \ \int_{\R^d}|z|^p \frac{1}{(2\pi)^{d/2}}\text{e}^{-\frac{1}{2}|z|^2}\,dz \ = \ \frac{2^{p/2}}{\Gamma(\frac{d}{2})}\int_0^\infty r^{\frac{p+d}{2}-1}\text{e}^{-r}\,dr \ = \ 2^{p/2}\frac{\Gamma(\frac{p+d}{2})}{\Gamma(\frac{d}{2})}.
\end{equation}
\end{proof}

\noindent Let $\mu^\sigma\coloneqq\mu*\Nc_\sigma$ and $\mu_N^\sigma\coloneqq\mu_N*\Nc_\sigma$. Then, both $\mu^\sigma$ and the random measure $\mu_N^\sigma$ are absolutely continuous with densities
\begin{equation}\label{Dens_g}
g^\sigma(x) \ = \ \int_{\R^d}\varphi^\sigma(x-y)\,\mu(dy), \qquad g_N^\sigma(x) \ = \ \frac{1}{N}\sum_{k=1}^N \varphi^\sigma(x-X_k), \qquad \forall\,x\in\R^d,
\end{equation}
where $\varphi^\sigma$ is given by \eqref{varphi} and $(X_k)_{k\geq1}$ is an i.i.d. sequence of random variables with $X_k\sim\mu$. Therefore, the Wasserstein distance between $\mu_N^\sigma$ and $\mu^\sigma$ admits an upper bound in terms of $g_N^\sigma$ and $g^\sigma$, as stated in the following lemma (see \cite[Theorem 1.6]{ChewiWeedRigollet}).

\begin{Lemma}\label{L:UpperBoundDensities}
Let $\mu,\nu\in\Pc_p(\R^d)$ be two probability measures with densities $f$ and $g$ respectively. Then, for every $p\geq1$, it holds
\[
\Wc_p^p(\mu,\nu) \ \leq \ 2^{p-1}\int_{\R^d} |x|^p\,|f(x)-g(x)|\,dx.
\]
\end{Lemma}

\noindent Next lemma concerns a Carlson-type inequality, which generalizes \cite[Lemma 2.1]{HK94}.

\begin{Lemma}\label{L:Carlson}
Let $g\colon\R^d\rightarrow\R$ be a non-negative, measurable function. Then, for $\beta>1$ and $\alpha>d(\beta-1)$, it holds
\begin{equation}\label{Carlson}
\int_{\R^d} g(x)\,dx \ \leq \ C_{\alpha,\beta,d} \bigg(\int_{\R^d} g(x)^\beta\,dx\bigg)^{\frac{\alpha-d(\beta-1)}{\alpha\beta}} \bigg(\int_{\R^d} |x|^\alpha\,g(x)^\beta\,dx\bigg)^{\frac{d(\beta-1)}{\alpha\beta}},
\end{equation}
with
\[
C_{\alpha,\beta,d} \ = \ \bigg(\frac{\alpha}{d(\beta-1)}\bigg)^{1/\beta}\bigg(\frac{d(\beta-1)}{\alpha-d(\beta-1)}\bigg)^{\frac{\alpha-d(\beta-1)}{\alpha\beta}}\bigg(\frac{2\pi^{d/2}}{\Gamma(\frac{d}{2})}\frac{1}{\alpha}I_{\alpha,\beta,d}\bigg)^{\frac{\beta-1}{\beta}},
\]
where
\[
I_{\alpha,\beta,d} \ = \ \int_0^\infty\frac{s^{\frac{d}{\alpha}-1}}{(1+s)^{\frac{1}{\beta-1}}}\,ds \ = \ \frac{\Gamma\big(\frac{d}{\alpha}\big) \Gamma\big(\frac{1}{\beta-1} - \frac{d}{\alpha}\big)}{\Gamma\big(\frac{1}{\beta-1}\big)}.
\]
\end{Lemma}
\begin{proof}
Let $t>0$. By H\"older's inequality, we get
\begin{align*}
\int_{\R^d} g(x)\,dx \ &= \ \int_{\R^d} \bigg(\frac{g(x)^\beta+g(x)^\beta t^\alpha |x|^\alpha}{1+t^\alpha|x|^\alpha}\bigg)^{1/\beta}dx \\
&\leq \ \bigg(\int_{\R^d}\frac{1}{(1+t^\alpha|x|^\alpha)^{1/(\beta-1)}}\,dx\bigg)^{\frac{\beta-1}{\beta}}\bigg(\int_{\R^d}\big(g(x)^\beta+g(x)^\beta t^\alpha|x|^\alpha\big)\,dx\bigg)^{1/\beta} \\
&= \ \frac{1}{t^{d\frac{\beta-1}{\beta}}}\bigg(\int_{\R^d}\frac{1}{(1+|x|^\alpha)^{1/(\beta-1)}}\,dx\bigg)^{\frac{\beta-1}{\beta}}\bigg(\int_{\R^d}\big(g(x)^\beta+g(x)^\beta t^\alpha|x|^\alpha\big)\,dx\bigg)^{1/\beta},
\end{align*}
where the last equality follows from the change of variable $x\rightarrow tx$ in the first integral. Now, denoting $F(t):=\frac{1}{t^{d\frac{\beta-1}{\beta}}}\big(\int_{\R^d}\big(g(x)^\beta+g(x)^\beta t^\alpha|x|^\alpha\big)\,dx\big)^{1/\beta}$, by \cite[formula 4.642]{GradshteinRyzhik} we find
\begin{align*}
\int_{\R^d} g(x)\,dx \ &\leq \ \bigg(\frac{2\pi^{d/2}}{\Gamma(\frac{d}{2})}\int_0^\infty\frac{r^{d-1}}{(1+r^\alpha)^{1/(\beta-1)}}\,dr\bigg)^{\frac{\beta-1}{\beta}}F(t) \\
&= \ \bigg(\frac{2\pi^{d/2}}{\Gamma(\frac{d}{2})}\frac{1}{\alpha}\int_0^\infty\frac{s^{\frac{d}{\alpha}-1}}{(1+s)^{1/(\beta-1)}}\,ds\bigg)^{\frac{\beta-1}{\beta}}F(t) \ = \ \bigg(\frac{2\pi^{d/2}}{\Gamma(\frac{d}{2})}\frac{1}{\alpha}I_{\alpha,\beta,d}\bigg)^{\frac{\beta-1}{\beta}}F(t),
\end{align*}
where the first equality follows from the change of variable $r\rightarrow s^{1/\alpha}$. Inequality \eqref{Carlson} follows by choosing $t^*$ minimizing the function $t\mapsto F(t)$, which is given by
\[
t^* \ = \ \bigg( \frac{d(\beta-1) \int_{\mathbb R^d} g(x)^\beta \, dx}{\big(\alpha - d(\beta-1)\big) \int_{\mathbb R^d} g(x)^\beta |x|^\alpha \, dx} \bigg)^{\frac{1}{\alpha}}.
\]
\end{proof}

\noindent The subsequent result is a consequence of the Marcinkiewicz-Zygmund inequality (see \cite[Theorem 2, Section 10.3]{ChowTeicher}).

\begin{Lemma}\label{L:MZ_Ineq}
Let $(\xi_k)_{k\geq1}$ be a sequence of independent identically distributed integrable real-valued random variables with $\mathbb E[\xi_k]=0$. Then, for $\beta\geq1$, if $\xi_1\in L^\beta$ it holds that
\begin{equation}\label{MZ_Ineq}
\bigg\|\sum_{k=1}^N \frac{\xi_k}{N}\bigg\|_{L^\beta} \ \leq \ \frac{C_\beta}{N^{\min\left(\frac{\beta-1}{\beta},\frac{1}{2}\right)}}\,\|\xi_1\|_{L^\beta},
\end{equation}
with $C_\beta=2\sqrt{[\beta/2]+1}$,  where $[\beta/2]$ is the integer part of $\beta/2$.
\end{Lemma}

\begin{Theorem}\label{T:RateSmooth}
Let $\mu\in\Pc(\R^d)$ and let $p\geq1$, $\sigma>0$. Assume that $M_q(\mu)<\infty$ for some $q>p$. Then, for any $1<\beta<\frac{q+d}{p+d}$,
\begin{equation}\label{RateSmooth}
\mathbb E\big[\big(\Wc_p^{(\sigma)}(\mu_N,\mu)\big)^p\big] \ \leq \ \frac{C_{\beta,\sigma}}{N^{\min\left(\frac{\beta-1}{\beta},\frac{1}{2}\right)}},
\end{equation}
where $C_{\beta,\sigma}$ is a non-negative constant depending on $p,q,d,\sigma,\beta$ and $M_q(\mu)$, see \eqref{EstimateFirstProof} and \eqref{EstimateSecondProof}.
\end{Theorem}
\begin{Remark}[On the exponent $\beta$ in \eqref{RateSmooth}]
Notice that, since $1<\beta<\frac{q+d}{p+d}$, the following inequalities hold:
\[
0 \ < \ \frac{\beta-1}{\beta} \ < \ \frac{q-p}{q+d}.
\]
Moreover, being $\beta<\frac{q+d}{p+d}$, we have
\begin{equation}\label{q>pbeta}
q \ > \ p\beta.
\end{equation}
\end{Remark}
\begin{Remark}[On the constant $C_{\beta,\sigma}$ in \eqref{RateSmooth}]
From estimate \eqref{EstimateFirstProof}, we obtain the following explicit formula for the constant $C_{\beta,\sigma}$ in \eqref{RateSmooth}:
\[
C_{\beta,\sigma} \ = \ \frac{2^pC_\beta C_{\alpha,\beta,d}}{(2\pi\sigma^2)^{d\frac{\beta-1}{2\beta}}}\bigg\{2^{q-1}M_q^q(\mu) + 2^{3q/2-1}\sigma^q\frac{\Gamma(\frac{q+d}{2})}{\Gamma(\frac{d}{2})}\bigg\}^{\frac{(p+d)\beta-d}{q\beta}},
\]
where $\alpha=q-p\beta$, $C_\beta$ is as in \eqref{MZ_Ineq}, $C_{\alpha,\beta,d}$ is as in \eqref{Carlson}. This formulation clarifies the dependence of $C_{\beta,\sigma}$ on $\sigma$. Specifically, there exist non-negative constants $C_1,C_2,C_3$, depending only on $p,q,d,\beta,M_q(\mu)$, such that
\[
C_{\beta,\sigma} \ = \ \frac{C_1}{\sigma^{d\frac{\beta-1}{\beta}}}\big\{C_2+C_3\sigma^q\big\}^{\frac{(p+d)\beta-d}{q\beta}}.
\]
\end{Remark}
\begin{Remark}\label{R:FG}
    Notice that $\Wc_p^{(\sigma)}(\mu,\nu)\leq\Wc_p(\mu,\nu)$, for every $\mu,\nu\in\mathcal P_p(\mathbb R^d)$. Thus using \cite[Theorem 1]{FG15} one also has the following upper bound, for some constant $C_{\textup{FG}}\geq0$, depending only on $p,d,q$:
    \[
    \mathbb E\big[\big(\Wc_p^{(\sigma)}(\mu_N,\mu)\big)^p\big] \leq C_{\textup{FG}}M_q^p(\mu)
    \begin{cases}
      \frac{1}{N^{\frac{1}{2}}} + \frac{1}{N^{\frac{q-p}{q}}}, \quad &\text{if }p>d/2\text{ and }q\neq2p, \\
      \frac{1}{N^{\frac{1}{2}}}\log(1+N) + \frac{1}{N^{\frac{q-p}{q}}}, \quad &\text{if }p=d/2\text{ and }q\neq2p, \\
      \frac{1}{N^{\frac{p}{d}}} + \frac{1}{N^{\frac{q-p}{q}}}, \quad &\text{if }p<d/2\text{ and }q\neq dp/(d-p). \\
    \end{cases}
    \]
    Observe that the above rates exhibit the curse of dimensionality when $p<d/2$ and $q>dp/(d-p)$, where the rate degrades to $N^{-p/d}$. By comparing these results with Theorem \ref{T:RateSmooth}, we see that Theorem \ref{T:RateSmooth} yields a better rate whenever $p\leq d/2$ and $q>2dp/(d-p)$. Moreover, this rate achieves the dimension-free bound $N^{-1/2}$ when $q>2p+d$ (note that if $p=d/2$, then $2dp/(d-p)=2p+d=4p$). Specifically:
\begin{itemize}
    \item If $p=d/2$, \cite[Theorem 1]{FG15} yields the exponent $1/2$ as soon as $q>2p$, but with an additional $\log(1+N)$ term; in contrast, Theorem \ref{T:RateSmooth} provides a $\log$-free $N^{-1/2}$ when $q>4p$ (recall that $2dp/(d-p)=4p$ when $p=d/2$);
    \item If $p<d/2$, \cite[Theorem 1]{FG15} yields the exponent $p/d$ when $q\geq dp/(d-p)$; however, Theorem \ref{T:RateSmooth} improves beyond $p/d$ precisely when $q>2dp/(d-p)$, reaching the dimension-free exponent $1/2$ when $q>2p+d$.
\end{itemize}
\end{Remark}

\begin{Remark}
We report two distinct proofs of Theorem \ref{T:RateSmooth}. However, it is worth noting that each proof provides a different constant, see \eqref{EstimateFirstProof} and \eqref{EstimateSecondProof}.
\end{Remark}
\begin{proof}[Proof (Theorem \ref{T:RateSmooth})]
We begin by stating the following estimate, which follows from \eqref{Dens_g} and Lemma \ref{L:UpperBoundDensities},
\begin{equation}\label{Estimate_W}
\mathbb E\big[\big(\Wc_p^{(\sigma)}(\mu_N,\mu)\big)^p\big] \ \leq \ 2^{p-1}\int_{\R^d} |x|^p\,\mathbb E\big[\big|g_N^\sigma(x)-g^\sigma(x)\big|\big]\,dx.
\end{equation}
In the sequel we denote by $X$ and $Z$ two independent random variables having distribution $\mu$ and $\Nc_1$, respectively.

\vspace{2mm}

\noindent\emph{1) Proof by Carlson-type inequality.} By \eqref{Estimate_W} and the Marcinkiewicz-Zygmund inequality \eqref{MZ_Ineq}, we have for $\beta>1$
\begin{align*}
\mathbb E\big[\big(\Wc_p^{(\sigma)}(\mu_N,\mu)\big)^p\big] \ &\leq \ 2^{p-1}\int_{\R^d} |x|^p\,\mathbb E\big[\big|g_N^\sigma(x)-g^\sigma(x)\big|^\beta\big]^{1/\beta}\,dx \\
&\leq \ 2^{p-1}\frac{C_\beta}{N^{\min\left(\frac{\beta-1}{\beta},\frac{1}{2}\right)}}\int_{\R^d} |x|^p\,\mathbb E\big[\big|\varphi^\sigma(x-X)-\mathbb E[\varphi^\sigma(x-X)]\big|^\beta\big]^{1/\beta}\,dx \\
&\leq \ 2^p\frac{C_\beta}{N^{\min\left(\frac{\beta-1}{\beta},\frac{1}{2}\right)}}\int_{\R^d} |x|^p\,\mathbb E\big[\big|\varphi^\sigma(x-X)\big|^\beta\big]^{1/\beta}\,dx.
\end{align*}
By Carlson-type inequality \eqref{Carlson}, we get, for $\beta>1$ such that $\beta<\frac{q+d}{p+d}$, $\alpha:=q-p\beta$ (recall from \eqref{q>pbeta} that $q>p\beta$) and $C_{p,\alpha,\beta,d}:=2^p C_\beta C_{\alpha,\beta,d}$ (where $C_{\alpha,\beta,d}$ is the constant of Carlson-type inequality \eqref{Carlson}),
\begin{align*}
&\mathbb E\big[\big(\Wc_p^{(\sigma)}(\mu_N,\mu)\big)^p\big] \notag \\
&\leq \ \frac{C_{p,\alpha,\beta,d}}{N^{\min\left(\frac{\beta-1}{\beta},\frac{1}{2}\right)}}\bigg(\int_{\R^d\times\R^d} |x|^{p\beta}|\varphi^\sigma(x-y)|^\beta\mu(dy)dx\bigg)^{\frac{\alpha-d(\beta-1)}{\alpha\beta}} \bigg(\int_{\R^d\times\R^d} |x|^q|\varphi^\sigma(x-y)|^\beta\mu(dy)dx\bigg)^{\frac{d(\beta-1)}{\alpha\beta}} \notag \\
&= \ \frac{C_{p,\alpha,\beta,d}}{N^{\min\left(\frac{\beta-1}{\beta},\frac{1}{2}\right)}}\bigg(\int_{\R^d\times\R^d} |x+y|^{p\beta}|\varphi^\sigma(y)|^\beta\mu(dx)dy\bigg)^{\frac{\alpha-d(\beta-1)}{\alpha\beta}} \bigg(\int_{\R^d\times\R^d} |x+y|^q|\varphi^\sigma(y)|^\beta\mu(dx)dy\bigg)^{\frac{d(\beta-1)}{\alpha\beta}} \notag \\
&= \ \frac{C_{p,\alpha,\beta,d}}{N^{\min\left(\frac{\beta-1}{\beta},\frac{1}{2}\right)}}\mathbb E\big[\big|X+\sigma Z\big|^{p\beta}\varphi^\sigma(\sigma Z)^{\beta-1}\big]^{\frac{\alpha-d(\beta-1)}{\alpha\beta}} \mathbb E\big[\big|X+\sigma Z\big|^q\varphi^\sigma(\sigma Z)^{\beta-1}\big]^{\frac{d(\beta-1)}{\alpha\beta}},
\end{align*}
where the second equality follows from the change of variable $(x,y)\rightarrow(x+y,x)$, while in the last equality we used that the random vector $(X,Y)=(X,\sigma Z)$ has distribution $\varphi^\sigma(y)dy\mu(dx)$. Since $\|\varphi^\sigma\|_\infty=(2\pi\sigma^2)^{-d/2}$, we find
\[
\mathbb E\big[\big(\Wc_p^{(\sigma)}(\mu_N,\mu)\big)^p\big] \ \leq \ \frac{C_{p,\alpha,\beta,d}}{N^{\min\left(\frac{\beta-1}{\beta},\frac{1}{2}\right)}}\frac{1}{(2\pi\sigma^2)^{d\frac{\beta-1}{2\beta}}}\mathbb E\big[\big|X+\sigma Z\big|^{p\beta}\big]^{\frac{\alpha-d(\beta-1)}{\alpha\beta}} \mathbb E\big[\big|X+\sigma Z\big|^q\big]^{\frac{d(\beta-1)}{\alpha\beta}}.
\]
By using Jensen's inequality, and then the elementary inequality $(a+b)^q\leq 2^{q-1}(a^q+b^q)$, for $a,b\geq0$, we obtain
\begin{align}
&\mathbb E\big[\big(\Wc_p^{(\sigma)}(\mu_N,\mu)\big)^p\big] \
\leq \ \frac{C_{p,\alpha,\beta,d}}{N^{\min\left(\frac{\beta-1}{\beta},\frac{1}{2}\right)}}\frac{1}{(2\pi\sigma^2)^{d\frac{\beta-1}{2\beta}}} \mathbb E\big[\big|X+\sigma Z\big|^q\big]^{\frac{p\beta}{q}\frac{\alpha-d(\beta-1)}{\alpha\beta} + \frac{d(\beta-1)}{\alpha\beta}} \notag \\
&= \ \frac{C_{p,\alpha,\beta,d}}{N^{\min\left(\frac{\beta-1}{\beta},\frac{1}{2}\right)}}\frac{1}{(2\pi\sigma^2)^{d\frac{\beta-1}{2\beta}}}\Big\{2^{q-1}\mathbb E\big[|X|^q\big] + 2^{q-1}\sigma^q\mathbb E\big[|Z|^q\big]\Big\}^{\frac{p\beta}{q}\frac{\alpha-d(\beta-1)}{\alpha\beta} + \frac{d(\beta-1)}{\alpha\beta}} \notag \\
&= \ \frac{C_{p,\alpha,\beta,d}}{N^{\min\left(\frac{\beta-1}{\beta},\frac{1}{2}\right)}}\frac{1}{(2\pi\sigma^2)^{d\frac{\beta-1}{2\beta}}}\bigg\{2^{q-1}M_q^q(\mu) + 2^{3q/2-1}\sigma^q\frac{\Gamma(\frac{q+d}{2})}{\Gamma(\frac{d}{2})}\bigg\}^{\frac{p\beta}{q}\frac{\alpha-d(\beta-1)}{\alpha\beta} + \frac{d(\beta-1)}{\alpha\beta}}, \label{EstimateFirstProof}
\end{align}
where the last equality follows from \eqref{Moments_N_1}.

\vspace{2mm}

\noindent\emph{2) Proof by dyadic partitioning argument.} We follow the dyadic partitioning
argument from \cite{FG15}. We begin by observing that, in the present context, since both measures $\mu_N*\Nc_\sigma$ and $\mu*\Nc_\sigma$ are absolutely continuous, the dyadic partitioning argument simplifies considerably. More precisely, the internal summation over $F$ in \cite[Lemma 6]{FG15} disappears, as it can be simply upper bounded by $\int_{B_n}|\frac{1}{N}\sum_{k=1}^N\varphi^\sigma(x-X_k)-\int_{\R^d}\varphi^\sigma(x-y)\mu(dy)|dx$, where $B_0=(-1,1]^d$ and $B_n=(-2^n,2^n]^d\backslash(-2^{n-1},2^{n-1}]^d$, for $n\geq1$. For this reason, only the summation over $n$ in \cite[Lemma 6]{FG15} becomes relevant. To upper bound $\mathbb E[(\Wc_p^{(\sigma)}(\mu_N,\mu))^p]$ by the summation over $n$, we proceed directly from \eqref{Estimate_W}:
\begin{align*}
\mathbb E\big[\big(\Wc_p^{(\sigma)}(\mu_N,\mu)\big)^p\big] \ &\leq \ 2^{p-1}\sum_{n\geq0}\int_{B_n} |x|^p\,\mathbb E\big[\big|g_N^\sigma(x)-g^\sigma(x)\big|\big]\,dx \\
&\leq \ 2^{p-1}d^{p/2}\sum_{n\geq0}2^{pn}\int_{B_n} \mathbb E\big[\big|g_N^\sigma(x)-g^\sigma(x)\big|\big]\,dx.
\end{align*}
Proceeding as at the beginning of the previous proof, by using the Marcinkiewicz-Zygmund inequality \eqref{MZ_Ineq}, we get, for $\beta\geq1$,
\begin{align*}
&\mathbb E\big[\big(\Wc_p^{(\sigma)}(\mu_N,\mu)\big)^p\big] \ \leq \ 2^{p-1}d^{p/2}\sum_{n\geq0}2^{pn}\int_{B_n} \mathbb E\big[\big|g_N^\sigma(x)-g^\sigma(x)\big|^\beta\big]^{1/\beta}\,dx \\
&\leq \ 2^p d^{p/2}\frac{C_\beta}{N^{\min\left(\frac{\beta-1}{\beta},\frac{1}{2}\right)}}\sum_{n\geq0}2^{pn}\int_{B_n} \mathbb E\big[\big|\varphi^\sigma(x-X)\big|^\beta\big]^{1/\beta}\,dx \\
&\leq \ 2^p d^{p/2}\frac{C_\beta}{N^{\min\left(\frac{\beta-1}{\beta},\frac{1}{2}\right)}}\sum_{n\geq0}2^{pn}\bigg(\int_{B_n} \mathbb E\big[\big|\varphi^\sigma(x-X)\big|^\beta\big]\,dx\bigg)^{1/\beta}\big(\lambda_d(B_n)\big)^{1-1/\beta},
\end{align*}
where $\lambda_d$ denotes the $d$-dimensional Lebesgue measure. Since $\lambda_d(B_n)\leq2^d2^{dn}$, it holds that
\begin{align*}
&\mathbb E\big[\big(\Wc_p^{(\sigma)}(\mu_N,\mu)\big)^p\big] \leq \frac{2^{p+d(1-1/\beta)} d^{p/2}C_\beta}{N^{\min\left(\frac{\beta-1}{\beta},\frac{1}{2}\right)}}\sum_{n\geq0}2^{pn+dn(1-1/\beta)}\bigg(\int_{B_n} \mathbb E\big[\big|\varphi^\sigma(x-X)\big|^\beta\big]\,dx\bigg)^{1/\beta} \\
&= \frac{2^{p+d(1-1/\beta)} d^{p/2}C_\beta}{N^{\min\left(\frac{\beta-1}{\beta},\frac{1}{2}\right)}}\sum_{n\geq0}2^{pn+dn(1-1/\beta)}\bigg(\int_{\R^d\times\R^d} \big|\varphi^\sigma(x-x')\big|^\beta\mathbbm{1}_{B_n}(x)\mu(dx')dx\bigg)^{1/\beta} \\
&= \frac{2^{p+d(1-1/\beta)} d^{p/2}C_\beta}{N^{\min\left(\frac{\beta-1}{\beta},\frac{1}{2}\right)}}\sum_{n\geq0}2^{pn+dn(1-1/\beta)}\bigg(\int_{\R^d\times\R^d} \big|\varphi^\sigma(x')\big|^\beta\mathbbm{1}_{B_n}(x+x')\mu(dx)dx'\bigg)^{1/\beta} \\
&= \frac{2^{p+d(1-1/\beta)} d^{p/2}C_\beta}{N^{\min\left(\frac{\beta-1}{\beta},\frac{1}{2}\right)}}\sum_{n\geq0}2^{pn+dn(1-1/\beta)}\,\mathbb E\Big[\big|\varphi^\sigma(\sigma Z)\big|^{\beta-1}\mathbbm{1}_{\{X+\sigma Z\in B_n\}}\Big]^{1/\beta},
\end{align*}
where the second equality follows from the change of variable $(x,x')\rightarrow(x+x',x)$, while in the last equality we used that the random vector $(X,X')=(X,\sigma Z)$ has distribution $\varphi^\sigma(x')dx'\mu(dx)$. Since $\|\varphi^\sigma\|_\infty=(2\pi\sigma^2)^{-d/2}$, we find
\begin{align*}
\mathbb E\big[\big(\Wc_p^{(\sigma)}(\mu_N,\mu)\big)^p\big] &\leq  \frac{2^{p+d(1-1/\beta)} d^{p/2}C_\beta}{N^{\min\left(\frac{\beta-1}{\beta},\frac{1}{2}\right)}(2\pi\sigma^2)^{d\frac{\beta-1}{2\beta}}}\sum_{n\geq0}2^{pn+dn(1-1/\beta)}\,\P(X+\sigma Z\in B_n)^{1/\beta} \\
&= \frac{2^{p+d(1-1/\beta)} d^{p/2}C_\beta}{N^{\min\left(\frac{\beta-1}{\beta},\frac{1}{2}\right)}(2\pi\sigma^2)^{d\frac{\beta-1}{2\beta}}}\sum_{n\geq0}2^{pn+dn(1-1/\beta)}\Big(\P\big(|X+\sigma Z|_\infty\leq2^n\big)\mathbbm{1}_{\{n=0\}} \\
&\quad + \P\big(2^{n-1}<|X+\sigma Z|_\infty\leq2^n\big)\mathbbm{1}_{\{n\geq1\}}\Big)^{1/\beta},
\end{align*}
where $|\cdot|_\infty$ denotes the infinity norm on $\R^d$. Then
\begin{align*}
&\mathbb E\big[\big(\Wc_p^{(\sigma)}(\mu_N,\mu)\big)^p\big] \\
&\leq \frac{2^{p+d(1-1/\beta)} d^{p/2}C_\beta}{N^{\min\left(\frac{\beta-1}{\beta},\frac{1}{2}\right)}(2\pi\sigma^2)^{d\frac{\beta-1}{2\beta}}}\sum_{n\geq0}2^{pn+dn(1-1/\beta)}\Big(\mathbbm{1}_{\{n=0\}} + \P\big(|X+\sigma Z|_\infty>2^{n-1}\big)\mathbbm{1}_{\{n\geq1\}}\Big)^{1/\beta} \\
&\leq \frac{2^{p+d(1-1/\beta)} d^{p/2}C_\beta}{N^{\min\left(\frac{\beta-1}{\beta},\frac{1}{2}\right)}(2\pi\sigma^2)^{d\frac{\beta-1}{2\beta}}}\sum_{n\geq0}2^{pn+dn(1-1/\beta)}\Big(\mathbbm{1}_{\{n=0\}} + \Big(\P\big(\sigma|Z|_\infty>2^{n-2}\big) \\
&\quad + \P\big(|X|_\infty>2^{n-2}\big)\Big)\mathbbm{1}_{\{n\geq1\}}\Big)^{1/\beta}.
\end{align*}
Recall that $Z\sim\Nc_1$. Denote by $Z^{(1)},\ldots,Z^{(d)}$ the components of $Z$, that is $Z=(Z^{(1)},\ldots,Z^{(d)})$. Then, 
\[
    \P(\sigma|Z|_\infty>2^{n-2}) \ \leq \ d\,\P(\sigma|Z^{(1)}|>2^{n-2}) \ = \ 2d\,\P(\sigma Z^{(1)}>2^{n-2}) \ \leq \ 2d\,\text{e}^{-\frac{1}{\sigma^2}2^{2n-5}},
\]
where the last inequality follows from the standard one-dimensional Gaussian tail upper bound
$\mathbb P(Z^{(1)}\ge u)\le \text{e}^{-u^2/2}$, for $u>0$ (see e.g. \cite[Theorem 2.1]{BLM13}).
Moreover, since $|X|_\infty\leq|X|$, where $|\cdot|$ denotes the Euclidean norm on $\R^d$, by Markov's inequality we get
\begin{align*}
&\mathbb E\big[\big(\Wc_p^{(\sigma)}(\mu_N,\mu)\big)^p\big] \\
&\leq \frac{2^{p+d(1-1/\beta)} d^{p/2}C_\beta}{N^{\min\left(\frac{\beta-1}{\beta},\frac{1}{2}\right)}(2\pi\sigma^2)^{d\frac{\beta-1}{2\beta}}}\sum_{n\geq0}2^{pn+dn(1-1/\beta)}\bigg(\!\mathbbm{1}_{\{n=0\}} + \bigg(2d\,\text{e}^{-\frac{1}{\sigma^2}2^{2n-5}} + \frac{M_q^q(\mu)}{2^{q(n-2)}}\!\bigg)\!\mathbbm{1}_{\{n\geq1\}}\!\bigg)^{1/\beta} \\
&\leq \frac{2^{p+d(1-1/\beta)} d^{p/2}C_\beta}{N^{\min\left(\frac{\beta-1}{\beta},\frac{1}{2}\right)}(2\pi\sigma^2)^{d\frac{\beta-1}{2\beta}}}\sum_{n\geq0}2^{pn+dn(1-1/\beta)}\bigg(\mathbbm{1}_{\{n=0\}} + \bigg(\!(2d)^{1/\beta}\text{e}^{-\frac{1}{\beta\sigma^2}2^{2n-5}} + \frac{M_q^{q/\beta}(\mu)}{2^{\frac{q}{\beta}(n-2)}}\bigg)\mathbbm{1}_{\{n\geq1\}}\!\bigg).
\end{align*}
In particular, the last inequality follows from $(a+b)^{1/\beta}\leq a^{1/\beta}+b^{1/\beta}$, for $a,b\geq0$ (recall that $\beta>1$). Since $\beta<\frac{q+d}{p+d}$, the above series converges and we end up with
\begin{align}
\mathbb E\big[\big(\Wc_p^{(\sigma)}(\mu_N,\mu)\big)^p\big] \ &\leq \ \frac{2^{p+d(1-1/\beta)} d^{p/2}C_\beta}{N^{\min\left(\frac{\beta-1}{\beta},\frac{1}{2}\right)}(2\pi\sigma^2)^{d\frac{\beta-1}{2\beta}}}\bigg\{\frac{2^{\frac{2q}{\beta}}}{1 - 2^{p+d - \frac{q+d}{\beta}}}M_q^{q/\beta} (\mu)\label{EstimateSecondProof} \\
&\quad \ + \sum_{n\geq0}2^{pn+dn(1-1/\beta)}\bigg(\mathbbm{1}_{\{n=0\}} + (2d)^{1/\beta}\text{e}^{-\frac{1}{\beta\sigma^2}2^{2n-5}}\mathbbm{1}_{\{n\geq1\}}\bigg)\bigg\}. \notag
\end{align}
\end{proof}

\noindent Having established in Theorem \ref{T:RateSmooth} that $\mathbb E[(\Wc_p^{(\sigma)}(\mu_N,\mu))^p]$ decays at a rate $O(N^{-1/2})$ whenever $\mu$ possesses sufficiently high moments ($q > 2p+d$), a natural question arises regarding the optimality of this result. As discussed in Remark \ref{R:FG}, this dimension-free rate represents a significant improvement over the classical Wasserstein distance, which suffers from the curse of dimensionality when $p < d/2$. Theorem \ref{T:SharpRate} establishes the existence of a probability measure for which the convergence rate is bounded from below by $N^{-1/2-\varepsilon}$, for any $\varepsilon > 0$.

\begin{Theorem}\label{T:SharpRate}
There exists $\mu\in \bigcap_{q\ge 1}\mathcal P_q(\mathbb R^d)$ such that, for any $\varepsilon>0$,
\[
\mathbb E\big[\big(\Wc_p^{(\sigma)}(\mu_N,\mu)\big)^p\big]\ \ge\  C\,N^{-1/2-\varepsilon},\qquad\text{for all }N\ge N_\varepsilon,
\]
for some constants $C>0$ (depending only on $p$) and $N_\varepsilon\in\N$ (depending only on $\varepsilon,\sigma,d$).
\end{Theorem}

\noindent For the proof of Theorem \ref{T:SharpRate} we need the following three technical results.

\begin{Lemma}[Neighborhood lower bound]\label{L:Neighborhood}
Let $p\ge 1$ and $\mu,\nu\in\mathcal P_p(\mathbb R^d)$. Then for every Borel set $B\subset\mathbb R^d$ and every $r>0$,
\[
\mathcal W_p(\mu,\nu)^p \ \ge\ r^p\big(\mu(B)-\nu(B^{(r)})\big)^+,
\]
where $(u)^+=\max\{u,0\}$ and $B^{(r)}$ denotes the closed $r$-neighborhood of $B$, that is $B^{(r)}:=\{x\in\mathbb R^d\colon \mathrm{dist}(x,B)\le r\}$, with $\mathrm{dist}$ denoting the Euclidean distance.
\end{Lemma}
\begin{proof}
Fix $B$ and $r>0$. Let $\pi\in\mathcal P_p(\mathbb R^d\times\mathbb R^d)$ be a coupling of $\mu$ and $\nu$. Then $\mu(B)=\pi(B\times \mathbb R^d)=\pi(B\times B^{(r)})+\pi(B\times (\mathbb R^d\setminus B^{(r)}))$. Moreover $\nu(B^{(r)})=\pi(\mathbb R^d\times B^{(r)})\ge\pi(B\times B^{(r)})$. Hence $\mu(B)-\nu(B^{(r)})\le \pi(B\times (\mathbb R^d\setminus B^{(r)}))$. On the set
$B\times (\mathbb R^d\setminus B^{(r)})$ we have $|x-y|>r$, so that
\[
\int_{\mathbb R^d\times\mathbb R^d}|x-y|^p\,\pi(dx,dy)
\ \ge\ r^p\,\pi\big(B\times (\mathbb R^d\setminus B^{(r)})\big)
\ \ge\ r^p\big(\mu(B)-\nu(B^{(r)})\big).
\]
Taking the positive part and then the infimum over all couplings $\pi$ yields the claim.
\end{proof}

\begin{Lemma}[Gaussian tail upper bound]\label{L:GaussianTail}
Let $\sigma Z\sim \mathcal N_\sigma$. Then for every $t>0$,
\[
\mathbb P(\sigma|Z|\ge t)\ \le\ 2d\,\textup{e}^{-\frac{t^2}{2d\sigma^2}}.
\]
\end{Lemma}
\begin{proof}
If $\sigma|Z|\ge t$ and $Z=(Z^{(1)},\ldots,Z^{(d)})$, then there exists some coordinate $j\in\{1,\dots,d\}$ such that $\sigma|Z^{(j)}|\ge t/\sqrt d$. Thus $\mathbb P(\sigma|Z|\ge t)\le \sum_{j=1}^d \mathbb P(\sigma|Z^{(j)}|\ge t/\sqrt d)
= d\,\mathbb P(\sigma|Z^{(1)}|\ge t/\sqrt d)$. Then, the claim follows from the standard one-dimensional Gaussian tail upper bound
$\mathbb P(|Z^{(1)}|\ge u)\le 2\text{e}^{-u^2/2}$, for $u>0$ (see e.g. \cite[Theorem 2.1]{BLM13}).
\end{proof}

\begin{Lemma}[On the binomial distribution]\label{L:Binomial}
There exist constants $v_0>0$ and $c_0>0$ such that the following property holds:
for every $n\in\mathbb N$ and every $p\in(0,1)$ satisfying $np(1-p)\ge v_0$,
if $X\sim\mathrm{Bin}(n,p)$, the binomial distribution with parameters $n$ and $p$, then
\[
\mathbb P\Big(X-np\ge \tfrac12\sqrt{np(1-p)}\Big)\ \ge\ c_0.
\]
\end{Lemma}
\begin{proof}
Let $(X_i)_{i=1}^n$ be i.i.d. Bernoulli$(p)$ random variables and let $X=\sum_{i=1}^n X_i$. Let $\Phi$ denote the standard normal cumulative distribution function. Let also $\rho:=\mathbb E[|X_1-p|^3]=p(1-p)^3+(1-p)p^3$. By the Berry-Esseen theorem, there exists a constant $C_{\mathrm{BE}}>0$ (independent of $n$ and $p$) such that, for all $x\in\mathbb R$,
\[
\bigg|\mathbb P\bigg(\frac{X-np}{\sqrt{np(1-p)}}\le x\bigg)-\Phi(x)\bigg|
\ \le \ \frac{C_{\mathrm{BE}}\,\rho}{\sqrt{np^3(1-p)^3}}
\ = \ \frac{C_{\mathrm{BE}}\big((1-p)^2+p^2\big)}{\sqrt{np(1-p)}}
\ \le \ \frac{C_{\mathrm{BE}}}{\sqrt{np(1-p)}}.
\]
Taking $x=1/2$ yields
\[
\mathbb P\bigg(\frac{X-np}{\sqrt{np(1-p)}} \ge \tfrac12\bigg)
\ \ge \ 1-\Phi(\tfrac12)-\frac{C_{\mathrm{BE}}}{\sqrt{np(1-p)}}.
\]
Choose $v_0$ large enough so that $C_{\mathrm{BE}}/\sqrt{v_0}\le (1-\Phi(1/2))/2$, and set
$c_0:=(1-\Phi(1/2))/2>0$. Then the above probability is at least $c_0$ whenever $np(1-p)\ge v_0$.
\end{proof}

\begin{proof}[Proof (of Theorem~\ref{T:SharpRate})]
We split the proof into six steps.

\vspace{2mm}

\noindent\textsc{Step I.} \emph{Construction of $\mu$.} Let $e_1=(1,0,\dots,0)\in\mathbb R^d$. For $k\ge 1$, set $a_k:=2^{-k^2}$ and $x_k:=2^k e_1$. Since $k^2\ge 2k$ for $k\ge 2$, we have $\sum_{k=1}^\infty a_k\le 2^{-1}+\sum_{k=2}^\infty 2^{-2k}<1$. Set $a_0:=1-\sum_{k=1}^\infty a_k>0$ and define the probability measure
\begin{equation}\label{eq:mu-def}
\mu \ := \ a_0\,\delta_0+\sum_{k=1}^\infty \frac{a_k}{2}\big(\delta_{x_k}+\delta_{-x_k}\big).
\end{equation}
Notice that $\mu$ has moments of all orders, i.e. $\mu\in \bigcap_{q\ge 1}\mathcal P_q(\mathbb R^d)$. In fact, fix $q\ge 1$. Using \eqref{eq:mu-def} and $|x_k|=2^k$, we get $\int_{\mathbb R^d}|x|^q\mu(dx)=\sum_{k=1}^\infty a_k |x_k|^q=\sum_{k=1}^\infty 2^{-(k^2-qk)}<\infty$.\\
Now, for $N\ge 16$ define $L_N:=\log_2 N$ and $k_N:=[\log_2 L_N]$, the integer part of $\log_2 L_N$. Then $2^{k_N}\le L_N<2^{k_N+1}$, hence
\begin{equation}\label{eq:2k-asymp-logN}
\frac{L_N}{2} \ \le \ 2^{k_N} \ \le \ L_N.
\end{equation}
Then, let $x_N:=x_{k_N}=2^{k_N}e_1$, $r_N:=2^{k_N-3}=2^{-3}|x_N|$, and let
$B_N:=\{x\in\mathbb R^d\colon |x-x_N|\leq r_N\}$, the ball of radius $r_N$ centered at $x_N$. Moreover, denote the mass at the atom $x_N$ by $w_N:=\mu(\{x_N\})=a_{k_N}/2=2^{-k_N^2-1}$. Finally, let $(X_i)_i$ be a sequence of i.i.d. random variables having distribution $\mu$. Then, define $S_N:=\sum_{i=1}^N \mathbf 1_{\{X_i=x_N\}}$ and $W_N:=\mu_N(\{x_N\})=S_N/N$, where we recall that $\mu_N = \frac{1}{N}\sum_{i=1}^N\delta_{X_i}$. Notice that $S_N\sim \mathrm{Bin}(N,w_N)$.

\vspace{2mm}

\noindent\textsc{Step II.} \emph{A lower bound for $\mu_N^\sigma(B_N)-\mu^\sigma(B_N^{(r_N)})$.} Write $\mu^\sigma:=\mu*\mathcal N_\sigma$ and $\mu_N^\sigma:=\mu_N*\mathcal N_\sigma$.
We first lower bound $\mu_N^\sigma(B_N)$ in terms of $W_N$. To this end, let $\sigma Z\sim \mathcal N_\sigma$ be independent of $(X_i)_i$. Set $\varepsilon_N:=\mathbb P(\sigma|Z|>r_N)$. Notice that
\[
\mathbb P(X_i+\sigma Z\in B_N\mid X_i=x_N) \ = \ \mathbb P(\sigma|Z|\le r_N) \ = \ 1-\varepsilon_N.
\]
Then, by \eqref{Dens_g} we find
\begin{equation}\label{eq:emp-mass-lb}
\mu_N^\sigma(B_N) \ = \ \frac{1}{N}\sum_{i=1}^N \int_{B_N}\varphi^\sigma(x-X_i)\,dx \ = \ \frac1N\sum_{i=1}^N \mathbb P(X_i+\sigma Z\in B_N\mid X_i)
\ \ge \ (1-\varepsilon_N)\,W_N,
\end{equation}
where in the last inequality we used that $\mathbb P(X_i+\sigma Z\in B_N\mid X_i\neq x_N)\geq0$. Next, we upper bound $\mu^\sigma(B_N^{(r_N)})$, where $B_N^{(r_N)}:=\{x\in\mathbb R^d\colon |x-x_N|\leq2r_N\}$ . Observe that $2r_N=2^{k_N-2}$. For any point $y$ in the support of $\mu$ other than $x_N$,
we have $|y-x_N|\ge 2^{k_N-1}$; indeed, $x_{k_N-1}=2^{k_N-1}e_1$ is the nearest distinct atom on the same ray, and all other atoms (including $0$ and the negative ones) are even farther. Therefore, for every $y\in\textup{supp}(\mu)\setminus\{x_N\}$, it holds that: $\mathrm{dist}\big(y,B_N^{(r_N)}\big)\ge 2^{k_N-1}-2r_N = 2^{k_N-2}$. Consequently, $\mathbb P(y+\sigma Z\in B_N^{(r_N)}) \le \mathbb P(\sigma|Z|\ge 2^{k_N-2})$. Using that the total $\mu$-mass outside $\{x_N\}$ is $1-w_N$, we obtain
\begin{equation}\label{eq:true-mass-ub}
\mu^\sigma(B_N^{(r_N)}) \ \le \ w_N + \mathbb P(\sigma|Z|\ge 2^{k_N-2}).
\end{equation}
Define $\delta_N:=\mathbb P(\sigma|Z|\ge 2^{k_N-2})$. Combining \eqref{eq:emp-mass-lb} and \eqref{eq:true-mass-ub}, yields the lower bound
\begin{equation}\label{eq:mass-diff}
\mu_N^\sigma(B_N)-\mu^\sigma(B_N^{(r_N)}) \ \ge \ (1-\varepsilon_N)W_N - w_N - \delta_N.
\end{equation}

\vspace{2mm}

\noindent\textsc{Step III.} \emph{On the probability $\mathbb P(W_N-w_N\ge c_1\sqrt{w_N/N})$.} We claim that there exist an integer $N_0\in\mathbb N$ (depending only on $d,\sigma$) such that, for all $N\ge N_0$,
\begin{equation}\label{eq:event-prob}
\mathbb P\bigg(W_N-w_N\ge c_1\sqrt{\frac{w_N}{N}}\bigg) \ \ge \ c_0,
\end{equation}
with $c_0$ as in Lemma \ref{L:Binomial} and $c_1:=\sqrt{3}/4$.
To prove this, note first that $w_N\le 1/4$ for all $N\ge 16$, hence $1-w_N\ge 3/4$.
Next observe that, by definition of $k_N$, $\log_2 (N w_N) = \log_2 N - (k_N^2+1)
\ge \log_2 N - (\log_2\log_2 N)^2 - 2 \rightarrow\infty$.
In particular, $N w_N(1-w_N)\to\infty$ as $N\to\infty$.\\
There exists $N_0\in\mathbb N$, independent of $\sigma,p,d$, such that $N w_N(1-w_N)\ge v_0$ for all $N\ge N_0$, where $v_0$ is from Lemma~\ref{L:Binomial}.
Then, by Lemma~\ref{L:Binomial} applied to $S_N\sim\mathrm{Bin}(N,w_N)$, we obtain
\[
\mathbb P\Big(S_N-Nw_N\ge \tfrac12\sqrt{Nw_N(1-w_N)}\Big) \ \ge \ c_0,
\qquad\text{for all }N\ge N_0.
\]
On this event, $W_N-w_N=\frac{S_N-Nw_N}{N}\ge \frac12\sqrt{\frac{w_N(1-w_N)}{N}}
\ge \frac{\sqrt3}{4}\sqrt{\frac{w_N}{N}}$, from which \eqref{eq:event-prob} follows.

\vspace{2mm}

\noindent\textsc{Step IV.} \emph{Upper bound for $\varepsilon_N+\delta_N$.} We now show that the error terms $\varepsilon_N$ and $\delta_N$ in \eqref{eq:mass-diff} are negligible compared to $\sqrt{w_N/N}$.
By Lemma~\ref{L:GaussianTail} and $r_N=2^{k_N-3}$, we have $\varepsilon_N=\mathbb P(\sigma|Z|>r_N)\le 2d\exp(-2^{2k_N-6}/(2d\sigma^2)) \le 2d\exp(-c\,4^{k_N})$, for a constant $c=c(d,\sigma)>0$. Similarly, $\delta_N=\mathbb P(\sigma|Z|\ge 2^{k_N-2})\le 2d\exp(-c\,4^{k_N})$. Using \eqref{eq:2k-asymp-logN}, we also have $4^{k_N}=(2^{k_N})^2\ge L_N^2/4$, hence $\varepsilon_N+\delta_N \le 4d\exp(-\frac{c}{4}(\log_2 N)^2)$. On the other hand, $\sqrt{w_N/N}= N^{-1/2}\,2^{-(k_N^2+1)/2}
= \exp(-\frac12\log N - \frac{\log 2}{2}(k_N^2+1))$, which is much larger than $4d\exp(-\frac{c}{4}(\log_2 N)^2)$ for large $N$. Therefore, there exists $N_1\ge N_0$ (depending only on $d,\sigma$) such that, for all $N\ge N_1$,
\begin{equation}\label{eq:error-small}
\varepsilon_N+\delta_N \ \le \ \frac{c_1}{4}\sqrt{\frac{w_N}{N}},
\end{equation}
where $c_1=\sqrt{3}/4$ is as in \textsc{Step}~III.

\vspace{2mm}

\noindent\textsc{Step V.} \emph{Lower bound for $\mathcal W_p^{(\sigma)}(\mu_N,\mu)$.} Fix $N\ge N_1$ and consider the event
\[
E_N:=\Big\{W_N-w_N\ge c_1\sqrt{\frac{w_N}{N}}\Big\}.
\]
By \eqref{eq:event-prob}, $\mathbb P(E_N)\ge c_0$, for all $N\ge N_1$.
On $E_N$, we have $W_N \ge w_N + c_1\sqrt{w_N/N}$.
Then, by \eqref{eq:mass-diff}, on the event $E_N$ we obtain
\[
\mu_N^\sigma(B_N)-\mu^\sigma(B_N^{(r_N)})
\ge (1-\varepsilon_N)\Big(w_N+c_1\sqrt{\frac{w_N}{N}}\Big)-w_N-\delta_N
= c_1\sqrt{\frac{w_N}{N}}-\varepsilon_N\Big(w_N+c_1\sqrt{\frac{w_N}{N}}\Big)-\delta_N.
\]
Since $w_N+c_1\sqrt{w_N/N}\le 1$, we further have
\[
\mu_N^\sigma(B_N)-\mu^\sigma(B_N^{(r_N)})
\ \ge \ c_1\sqrt{\frac{w_N}{N}}-(\varepsilon_N+\delta_N)
\ \ge \ \frac{3c_1}{4}\sqrt{\frac{w_N}{N}}
\ \ge \ \frac{c_1}{2}\sqrt{\frac{w_N}{N}},
\]
for all $N\ge N_1$, where we used \eqref{eq:error-small} in the second inequality.
Applying Lemma~\ref{L:Neighborhood}, gives
\[
\mathcal W_p(\mu_N^\sigma,\mu^\sigma)^p \
\ge \ r_N^p\Big(\mu_N^\sigma(B_N)-\mu^\sigma(B_N^{(r_N)})\Big)^+
\ \ge \ \frac{c_1}{2}\,r_N^p\sqrt{\frac{w_N}{N}},
\qquad\text{on }E_N.
\]
Taking expectations, using $\mathbb P(E_N)\ge c_0$ and $\mathcal W_p^{(\sigma)}(\mu_N,\mu)=\mathcal W_p(\mu_N^\sigma,\mu^\sigma)$, yields, for all $N\ge N_1$,
\begin{equation}\label{eq:main-lb-N}
\mathbb E\Big[\big(\mathcal W_p^{(\sigma)}(\mu_N,\mu)\big)^p\Big] \
\ge \ \frac{c_1c_0}{2}\,r_N^p\sqrt{\frac{w_N}{N}}
\ = \ C\,N^{-1/2}\,2^{k_N p-\frac12 k_N^2},
\end{equation}
where $C:=\frac{c_1c_0}{2}\,2^{-3p-1/2}>0$, depends only on $p$.

\vspace{2mm}

\noindent\textsc{Step VI.} \emph{From \eqref{eq:main-lb-N} to $N^{-1/2-\varepsilon}$.} Fix $\varepsilon>0$. Because $(\log\log N)^2=o(\log N)$ as $N\to\infty$, there exists $N_\varepsilon\ge N_1$ such that for all $N\ge N_\varepsilon$, $\frac12(\log_2\log_2 N)^2 \le \varepsilon\,\log_2 N$. Since $k_N\le \log_2\log_2 N$, this implies $2^{-k_N^2/2}\ge N^{-\varepsilon}$ and therefore
$2^{k_N p-k_N^2/2}\ge N^{-\varepsilon}$ for all $N\ge N_\varepsilon$.
Thus, for all $N\ge N_\varepsilon$,
\[
\mathbb E\Big[\big(\mathcal W_p^{(\sigma)}(\mu_N,\mu)\big)^p\Big]
 \ \ge \ C\,N^{-1/2}\,2^{k_N p-\frac12 k_N^2}
 \ \ge \ C\,N^{-1/2-\varepsilon}.
\]
\end{proof}

\section{Wasserstein distance}

This section investigates the rate of convergence for the expected $p$-th power of the $p$-Wasserstein distance $\mathbb E[\mathcal W_p^p(\mu_N,\mu)]$ when $\mu\in\mathcal P(\mathbb R^d)$ satisfies  $M_p(\mu)<\infty$, but no moment of order $q>p$ is assumed, that is $M_q(\mu)=\infty$ for every $q>p$. To clarify the presentation, in Remark \ref{R:RateG2} and Corollary \ref{C:Zygmund} we consider the specific case where $\mu$ satisfies, for some $p\in[1,\infty)$ and $\alpha>0$,
\begin{equation}\label{Zygmund}
\int_{\mathbb R^d}|x|^p\big(\log(1+|x|)\big)^\alpha\,\mu(dx) \ < \ \infty.
\end{equation}
If $X\sim\mu$, then $X$ belongs to the Zygmund space $L^p(\log L)^\alpha$. In such a context, we obtain the rate of convergence \eqref{UpperBoundWassZygmund}.\\
We begin with the following technical result.
\begin{Lemma}\label{L:RateG}
Let $p\geq1$ and let $\mu\in\mathcal P(\mathbb R^d)$ be such that $M_p(\mu)<\infty$, but $M_q(\mu)=\infty$ for every $q>p$. Then, there exists a Borel measurable function $G_\mu\colon[0,+\infty)\rightarrow[0,+\infty)$, satisfying:
\begin{enumerate}[\upshape 1)]
\item $G_\mu>0$ on $(0,+\infty)$;
\item $\int_{\mathbb R^d}G_\mu(|x|)\mu(dx)<+\infty$;
\item $\lim_{t\rightarrow+\infty}\frac{G_\mu(t)}{t^p}=+\infty$;
\item $t\mapsto\frac{G_\mu(t)}{t^p}$ is non-decreasing on $[0,+\infty)$;
\item it holds that
\[
\lim_{t\rightarrow+\infty}\frac{G_\mu(t)}{t^q} \ = \ 0, \qquad\text{if }q>p.
\]
\end{enumerate}
Finally, if $G_\mu\colon[0,+\infty)\rightarrow[0,+\infty)$ is a Borel measurable function satisfying items 1)-2)-3)-4)-5), then, for every $c>0$, it holds that
\begin{equation}\label{UI_p}
\mathbb E\Big[|X|^p \mathbbm{1}_{\{|X|\geq c\}}\Big] \ \leq \ \frac{c^p}{G_\mu(c)} \mathbb E\Big[G_\mu(|X|)\mathbbm{1}_{\{|X|\geq c\}}\Big] \ \leq \ \frac{c^p}{G_\mu(c)} \mathbb E[G_\mu(|X|)].
\end{equation}
\end{Lemma}
\begin{proof}
Let $X\sim\mu$ and consider the function $H_\mu\colon [0, +\infty) \to (0, +\infty)$ defined as follows:
\begin{equation}\label{H_mu}
H_\mu(t) \ = \ \int_t^{+\infty} p s^{p-1} \mathbb{P}(|X| > s) \, ds, \qquad t\geq0.
\end{equation}
Note that $H_\mu(0) = \mathbb{E}[|X|^p] < +\infty$. Since $X \notin L^q$ for any $q > p$, the random variable $X$ is essentially unbounded, so $\mathbb{P}(|X| > s) > 0$ for all $s$, implying $H_\mu(t) > 0$ for all $t \ge 0$. Now, we define the function $G_\mu\colon [0, +\infty) \to [0, +\infty)$ as
\begin{equation}\label{G_mu}
G_\mu(t) \  = \int_0^t\frac{p s^{p-1}}{\sqrt{H_\mu(s)}} \, ds, \qquad t\geq0.
\end{equation}
It can be easily verified that this function satisfies all the above five conditions.\\
It remains to prove \eqref{UI_p}. Let $c>0$ be fixed and let $a>0$ be such that $\frac{G_\mu(t)}{t^p}\geq a$ for $t\geq c$. Then, $|X|^p\leq G_\mu(|X|)/a$ on $\{|X|\geq c\}$, and so
\[
\mathbb E\Big[|X|^p\mathbbm{1}_{\{|X|\geq c\}}\Big] \ \leq \ \frac{1}{a} \mathbb E\Big[G_\mu(|X|)\mathbbm{1}_{\{|X|\geq c\}}\Big].
\]
Then \eqref{UI_p} follows by taking $a=\frac{G_\mu(c)}{c^p}$, which is possible due to the non-decreasing monotonicity of $t\mapsto\frac{G_\mu(t)}{t^p}$.
\end{proof}

\begin{Remark}\label{R:RateG2}
Let $p\geq1$ and $\alpha>0$. Let $\mu\in\mathcal P(\mathbb R^d)$ be such that \eqref{Zygmund} holds, that is $X\sim\mu$ belongs to the Zygmund space $L^p(\log L)^\alpha$. In such a case, the function
\begin{equation}\label{G_muZygmund}
G_\mu(t) \ = \ t^p\big(\log(1+t)\big)^\alpha, \qquad t\geq0,
\end{equation}
satisfies all the five conditions of Lemma \ref{L:RateG}. In such a context, using the function $G_\mu$ in \eqref{G_muZygmund} we can find the rate of convergence \eqref{UpperBoundWassZygmund}.
\end{Remark}

\begin{Theorem}\label{T:RateG}
Let $\mu\in\Pc(\mathbb R^d)$ be such that $M_p(\mu)<\infty$ for some $p\geq1$ and let $X\sim\mu$. Let $G_\mu\colon[0,+\infty)\rightarrow[0,+\infty)$ be a Borel measurable function satisfying items 1)-2)-3)-4)-5) of Lemma \ref{L:RateG}. Then, for every $\varepsilon>0$, it holds that
\begin{equation}\label{UpperBoundWass}
\mathbb E\big[\Wc_p^p(\mu_N,\mu)\big] \ \leq \ C\Big(1 + \mathbb E\big[G_\mu(|X|)\big]\Big)\frac{N^{p\gamma_\varepsilon}}{G_\mu(N^{\gamma_\varepsilon})} + O\bigg(\frac{N^{p\gamma_\varepsilon}}{G_\mu(N^{\gamma_\varepsilon})}\bigg(\frac{G_\mu(N^{\gamma_\varepsilon})}{N^{(p+\varepsilon)\gamma_\varepsilon}}\bigg)^{\frac{p+d}{p}}\bigg),
\end{equation}
where $\gamma_\varepsilon=\frac{p}{2(p+\varepsilon)(p+d)}$ and $C=\max\{2^{p-1}C_{p,d}^p,2^{3p-2}\}$, with $C_{p,d}=2^{3/2}\left(\frac{\Gamma(\frac{p+d}{2})}{\Gamma(\frac{d}{2})}\right)^{1/p}$.
\end{Theorem}
\begin{Remark}
By condition 5) of Lemma \ref{L:RateG}, it follows that
\[
\frac{N^{p\gamma_\varepsilon}}{G_\mu(N^{\gamma_\varepsilon})}\bigg(\frac{G_\mu(N^{\gamma_\varepsilon})}{N^{(p+\varepsilon)\gamma_\varepsilon}}\bigg)^{\frac{p+d}{p}} \ = \ o\bigg(\frac{N^{p\gamma_\varepsilon}}{G_\mu(N^{\gamma_\varepsilon})}\bigg).
\]
\end{Remark}
\begin{proof}
We begin by finding a suitable upper bound for the Gaussian-smoothed Wasserstein distance between $\mu_N$ and $\mu$. Fix $\varepsilon>0$ and let $\gamma=\frac{p}{2(p+\varepsilon)(p+d)}$(to alleviate the presentation we denote $\gamma_\varepsilon$ simply as $\gamma$). By \eqref{Dens_g}, Lemma \ref{L:UpperBoundDensities}, and using that $g^\sigma(x)=\mathbb E[g_N^\sigma(x)]$, we have
\begin{align*}
&\mathbb E\big[\big(\Wc_p^{(\sigma)}(\mu_N,\mu)\big)^p\big] \ \leq \ 2^{p-1}\int_{\mathbb R^d} |x|^p\,\mathbb E\big[\big|g_N^\sigma(x)-g^\sigma(x)\big|\big]\,dx \\
&= \ 2^{p-1}\int_{\mathbb R^d} |x|^p\,\mathbb E\big[\big|g_N^\sigma(x)-g^\sigma(x)\big|\big]\,\big(\mathbbm{1}_{\{|x|\leq 2N^\gamma\}}+\mathbbm{1}_{\{|x|> 2N^\gamma\}}\big)\,dx \\
&\leq \ 2^{p-1}\int_{\mathbb R^d} |x|^p\,\sqrt{\text{Var}\big(g_N^\sigma(x)\big)}\,\mathbbm{1}_{\{|x|\leq 2N^\gamma\}}\,dx 
 + 2^p\int_{\mathbb R^d}|x|^p\bigg(\int_{\mathbb R^d}\varphi^\sigma(x-y)\,\mu(dy)\bigg)\mathbbm{1}_{\{|x|>2N^\gamma\}}\,dx\qquad \\
&\leq \ \frac{2^{p-1}}{\sqrt{N}}\int_{\mathbb R^d} |x|^p\,\sqrt{\text{Var}(\varphi^\sigma(x-X))}\,\mathbbm{1}_{\{|x|\leq 2N^\gamma\}}\,dx + 2^p\int_{\mathbb R^d\times\mathbb R^d}|y+z|^p\,\varphi^\sigma(z)\,\mathbbm{1}_{\{|y+z|>2N^\gamma\}}\,\mu(dy)\,dz,
\end{align*}
where $X\sim\mu$ and in the last integral we performed the change of variable $(x,y)\rightarrow(y+z,y)$. Now, let $Z\sim\Nc_1$ be independent of $X$. Then, the last integral can be written as $\mathbb E[|X+\sigma Z|^p\mathbbm 1_{\{|X+\sigma Z|>2N^\gamma\}}]$. Now, using that $\sqrt{\text{Var}(\varphi^\sigma(x-X))}\leq\sqrt{\mathbb E[(\varphi^\sigma(x-X))^2]}\leq(2\pi\sigma^2)^{-d/2}$, we find
\begin{align*}
&\mathbb E\big[\big(\Wc_p^{(\sigma)}(\mu_N,\mu)\big)^p\big] \ \leq \ \frac{2^{p-1}}{\sqrt{N}}\frac{1}{(2\pi)^{d/2}\sigma^d}\int_{\mathbb R^d}|x|^p\,\mathbbm{1}_{\{|x|\leq 2N^\gamma\}}\,dx \ + \ 2^p\,\mathbb E\big[|X+\sigma Z|^p\,\mathbbm{1}_{\{|X+\sigma Z|>2N^\gamma\}}\big] \\
&\leq \ \frac{2^{p-1}}{\sqrt{N}}\frac{1}{(2\pi)^{d/2}\sigma^d}\int_{\mathbb R^d}|x|^p\,\mathbbm{1}_{\{|x|\leq 2N^\gamma\}}\,dx \ + \ 2^{2p-1}\Big\{\mathbb E\big[|X|^p\,\mathbbm{1}_{\{|X|>N^\gamma\}}\big] + \mathbb E\big[|X|^p\,\mathbbm{1}_{\{\sigma|Z|>N^\gamma\}}\big] \\
&\quad \ + \sigma^p\mathbb E\big[|Z|^p\,\mathbbm{1}_{\{|X|>N^\gamma\}}\big] + \sigma^p\mathbb E\big[|Z|^p\,\mathbbm{1}_{\{\sigma|Z|>N^\gamma\}}\big]\Big\} \\
&= \ \frac{1}{\sqrt{N}}\frac{2^{2p+d/2}}{\sigma^d(d+p)\Gamma(\frac{d}{2})} N^{\gamma(p+d)} + 2^{2p-1}\Big\{\mathbb E\big[|X|^p\,\mathbbm{1}_{\{|X|>N^\gamma\}}\big] + M_p^p(\mu)\,\P(\sigma|Z|>N^\gamma) \\
&\quad \ + \sigma^p M_p^p(\Nc_1)\,\P(|X|>N^\gamma) + \sigma^p\mathbb E\big[|Z|^p\,\mathbbm{1}_{\{\sigma|Z|>N^\gamma\}}\big]\Big\},
\end{align*}
where in the second inequality we used $(a+b)^p\leq 2^{p-1}(a^p+b^p)$, $a,b\geq0$; on the other hand, in the last equality we used that both expectations $\mathbb E[|X|^p\mathbbm 1_{\{\sigma|Z|>N^\gamma\}}]$, $\mathbb E[|Z|^p\mathbbm 1_{\{|X|>N^\gamma\}}]$ factor as a consequence of the independence of $X$ and $Z$, moreover we wrote $\int_{\mathbb R^d}|x|^p\mathbbm{1}_{\{|x|\leq 2N^\gamma\}}dx=\int_0^{2N^\gamma} r^{p+d-1} \frac{2\pi^{d/2}}{\Gamma(\frac{d}{2})}dr=\frac{2^{1+p+d}\pi^{d/2} N^{\gamma(p+d)}}{(p+d)\Gamma(\frac{d}{2})}$. By using H\"older's inequality, we find, for every $p'>p$,
\begin{align*}
&\mathbb E\big[\big(\Wc_p^{(\sigma)}(\mu_N,\mu)\big)^p\big] \\
&\leq \ \frac{1}{\sqrt{N}}\frac{2^{2p+d/2}}{\sigma^d(d+p)\Gamma(\frac{d}{2})} N^{\gamma(p+d)}  + 2^{2p-1}\Big\{\mathbb E\big[|X|^p\,\mathbbm{1}_{\{|X|>N^\gamma\}}\big] + M_p^p(\mu)\,\P(\sigma|Z|>N^\gamma) \\
&\quad \ + \sigma^p M_p^p(\Nc_1)\,\P(|X|>N^\gamma) + \sigma^p M_{p'}^p(\Nc_1)\,\P(\sigma|Z|>N^\gamma)^{1-p/p'}\Big\} \\
&= \ \frac{1}{\sqrt{N}}\frac{2^{2p+d/2}}{\sigma^d(d+p)\Gamma(\frac{d}{2})} N^{\gamma(p+d)} + 2^{2p-1}\bigg\{\mathbb E\big[|X|^p\,\mathbbm{1}_{\{|X|>N^\gamma\}}\big] + M_p^p(\mu)\,\P(\sigma|Z|>N^\gamma) \\
&\quad \ + \sigma^p 2^{p/2}\frac{\Gamma(\frac{p+d}{2})}{\Gamma(\frac{d}{2})}\P(|X|>N^\gamma) + \sigma^p2^{p/2}\bigg(\frac{\Gamma(\frac{p'+d}{2})}{\Gamma(\frac{d}{2})}\bigg)^{p/p'}\P(\sigma|Z|>N^\gamma)^{1-p/p'}\bigg\},
\end{align*}
where the last equality follows from \eqref{Moments_N_1}. Then, by using Lemma \ref{L:GaussianTail}, \eqref{UI_p}, and Markov's inequality, we get
\begin{align*}
&\mathbb E\big[\big(\Wc_p^{(\sigma)}(\mu_N,\mu)\big)^p\big] \ \leq \ \frac{1}{\sqrt{N}}\frac{2^{2p+d/2}}{\sigma^d(d+p)\Gamma(\frac{d}{2})} N^{\gamma(p+d)} + 2^{2p-1}\bigg\{\frac{N^{p\gamma}}{G_\mu(N^\gamma)}\mathbb E[G_\mu(|X|)] \\
&+ 2d M_p^p(\mu)\text{e}^{-\frac{1}{2d\sigma^2}N^{2\gamma}} + \sigma^p2^{p/2}\frac{\Gamma(\frac{p+d}{2})}{\Gamma(\frac{d}{2})}\frac{\mathbb E[G_\mu(|X|)]}{G_\mu(N^\gamma)} + \sigma^p2^{p/2}\bigg(\frac{\Gamma(\frac{p'+d}{2})}{\Gamma(\frac{d}{2})}\bigg)^{p/p'}(2d)^{1-p/p'}\text{e}^{-\frac{1-p/p'}{2d\sigma^2}N^{2\gamma}}\bigg\}.
\end{align*}
Now, let us focus on the Wasserstein distance. Recalling Lemma \ref{L:Bound}, we find
\begin{align*}
&\mathbb E\big[\Wc_p^p(\mu_N,\mu)\big] \ \leq \ 2^{p-1}C_{p,d}^p\,\sigma^p+2^{p-1}\mathbb E\big[\big(\Wc_p^{(\sigma)}(\mu_N,\mu)\big)^p\big] \ \leq \ 2^{p-1}C_{p,d}^p\,\sigma^p \\
&+ \frac{1}{\sqrt{N}}\frac{2^{3p+d/2-1}}{\sigma^d(d+p)\Gamma(\frac{d}{2})} N^{\gamma(p+d)} + 2^{3p-2}\bigg\{\frac{N^{p\gamma}}{G_\mu(N^\gamma)}\mathbb E[G_\mu(|X|)] + 2dM_p^p(\mu)\text{e}^{-\frac{1}{2d\sigma^2}N^{2\gamma}} \\
&+ \sigma^p2^{p/2}\frac{\Gamma(\frac{p+d}{2})}{\Gamma(\frac{d}{2})}\frac{\mathbb E[G_\mu(|X|)]}{G_\mu(N^\gamma)} + \sigma^p2^{p/2}\bigg(\frac{\Gamma(\frac{p'+d}{2})}{\Gamma(\frac{d}{2})}\bigg)^{p/p'}(2d)^{1-p/p'}\text{e}^{-\frac{1-p/p'}{2d\sigma^2}N^{2\gamma}}\bigg\}.
\end{align*}
Let us set $\sigma=\frac{N^\gamma}{G_\mu(N^\gamma)^{1/p}}$. Then, we obtain, for any $p'>p$,
\begin{align}
&\mathbb E\big[\mathcal W_p^p(\mu_N,\mu)\big] \notag \\
&\leq \ 2^{p-1}C_{p,d}^p\,\frac{N^{p\gamma}}{G_\mu(N^\gamma)} + \frac{N^{p\gamma}G_\mu(N^\gamma)^{d/p}}{\sqrt{N}}\frac{2^{3p+d/2-1}}{(d+p)\Gamma(\frac{d}{2})} + 2^{3p-2}\bigg\{\frac{N^{p\gamma}}{G_\mu(N^\gamma)}\mathbb E[G_\mu(|X|)] \notag \\
&\quad \ + 2dM_p^p(\mu)\text{e}^{-\frac{1}{2d}G_\mu(N^\gamma)^{2/p}} + 2^{p/2}\frac{N^{p\gamma}}{G_\mu(N^\gamma)^2}\frac{\Gamma(\frac{p+d}{2})}{\Gamma(\frac{d}{2})}\mathbb E[G_\mu(|X|)] \label{lot} \\
&\quad \ + \frac{N^{p\gamma}}{G_\mu(N^\gamma)}2^{p/2}\bigg(\frac{\Gamma(\frac{p'+d}{2})}{\Gamma(\frac{d}{2})}\bigg)^{p/p'}(2d)^{1-p/p'}\text{e}^{-\frac{1-p/p'}{2d}G_\mu(N^\gamma)^{2/p}}\bigg\} \notag \\
&= \ \Big(2^{p-1}C_{p,d}^p + 2^{3p-2}\mathbb E[G_\mu(|X|)]\Big)\frac{N^{p\gamma}}{G_\mu(N^\gamma)} + \frac{N^{p\gamma}G_\mu(N^\gamma)^{d/p}}{\sqrt{N}}\frac{2^{3p+d/2-1}}{(d+p)\Gamma(\frac{d}{2})} + o\bigg(\frac{N^{p\gamma}}{G_\mu(N^\gamma)}\bigg). \notag
\end{align}
Recalling that $\gamma=\frac{p}{2(p+\varepsilon)(p+d)}$, we obtain
\[
\frac{N^{p\gamma}G_\mu(N^\gamma)^{d/p}}{\sqrt{N}} \ = \ \frac{N^{p\gamma}}{G_\mu(N^\gamma)}\frac{G_\mu(N^\gamma)^{\frac{p+d}{p}}}{\sqrt{N}} \ = \ \frac{N^{p\gamma}}{G_\mu(N^\gamma)}\bigg(\frac{G_\mu(N^\gamma)}{N^{(p+\varepsilon)\gamma}}\bigg)^{\frac{p+d}{p}},
\]
from which the claim follows.
\end{proof}

\begin{Corollary}\label{C:Zygmund}
Let $p\geq1$ and $\alpha>0$. Let $\mu\in\mathcal P(\mathbb R^d)$ be such that \eqref{Zygmund} holds, that is $X\sim\mu$ belongs to the Zygmund space $L^p(\log L)^\alpha$. Then, for every $\varepsilon>0$, it holds that
\begin{equation}\label{UpperBoundWassZygmund}
\mathbb E\big[\Wc_p^p(\mu_N,\mu)\big] \ \leq \ C\frac{1 + \mathbb E\big[|X|^p\big(\log(1+|X|)\big)^\alpha\big]}{(\log(1+N^{\gamma_\varepsilon}))^\alpha} + o\bigg(\frac{1}{(\log(1+N^{\gamma_\varepsilon}))^\alpha}\bigg),
\end{equation}
where $\gamma_\varepsilon=\frac{p}{2(p+\varepsilon)(p+d)}$ and $C=\max\{2^{p-1}C_{p,d}^p,2^{3p-2}\}$, with $C_{p,d}=2^{3/2}\left(\frac{\Gamma(\frac{p+d}{2})}{\Gamma(\frac{d}{2})}\right)^{1/p}$.
\end{Corollary}
\begin{proof}
The claim follows directly from Theorem \ref{T:RateG} taking $G_\mu$ as in \eqref{G_muZygmund}.
\end{proof}

\vspace{5mm}

\footnotesize
\noindent\textbf{Acknowledgements.} \ A. Cosso acknowledges support from GNAMPA-INdAM (of which L. Perelli and M. Martini are also members), the MUR project PRIN 2022 ``Entropy martingale optimal transport and McKean-Vlasov equations'', and the MUR project PRIN 2022 PNRR ``Probabilistic methods for energy transition''. M. Martini acknowledges the financial support of the European Research Council (ERC) under the European Union's Horizon Europe research and innovation program (AdG ELISA project, Grant Agreement No. 101054746). Views and opinions expressed are however those of the author(s) only and do not necessarily reflect those of the European Union or the European Research Council Executive Agency. Neither the European Union nor the granting authority can be held responsible for them.

\small
\bibliographystyle{plain}
\bibliography{references}

\end{document}